\documentclass[10pt]{amsart}
\usepackage[dvipsnames]{xcolor}
\usepackage[english]{babel}
\usepackage{amsfonts}
\usepackage{amsthm}
\usepackage{amsbsy}
\usepackage{amssymb}
\usepackage{graphicx}
\usepackage{eso-pic}
\usepackage{tikz} \usetikzlibrary{arrows,decorations.markings,matrix}
\usepackage{xfrac}
\usepackage{float}
\usepackage[bookmarks=true, colorlinks=true, urlcolor=Blue, linkcolor=BrickRed,citecolor=MidnightBlue, pagebackref=true]{hyperref}
\usepackage{bookmark}
\usepackage{amsmath, amscd}
\usepackage{subfigure}
\usepackage{scalerel,stackengine}
\usepackage{epigraph}
\usepackage{placeins}
\usepackage[marginparwidth=2cm]{geometry}
\usepackage{nicefrac}
\usepackage{todonotes}

\newtheorem{lemma}{Lemma}[section]

\newtheorem*{thm*}{Theorem}
\newtheorem{prop}[lemma]{Proposition}

\newtheorem*{cor*}{Corollary}

\theoremstyle{definition}

\theoremstyle{remark}
\newtheorem{rem}[lemma]{Remark}

\newcommand{\matR} {\ensuremath {\mathbb{R}}}
\newcommand{\matQ} {\ensuremath {\mathbb{Q}}}
\newcommand{\matZ} {\ensuremath {\mathbb{Z}}}

\newcommand{\matH} {\ensuremath {\mathbb{H}}}

\newcommand{\Sym}{\ensuremath {\mathfrak{S}}}
\newcommand{\Alt}{\ensuremath {\mathfrak{A}}}

\newcommand{\Isom}{\ensuremath {\mathrm{Isom}}}
\newcommand{\orb}{{\mathrm{orb}}}
\newcommand{\Vol}{{\mathrm{Vol}}}

\newcommand{\id}{\mathrm{id}}

\newcommand{\Aut}{\mathrm{Aut}}
\newcommand{\diag}{\mathrm{diag}}

\newcommand{\w}[1]{{\color{white} #1}}

\makeatletter
\newcommand\xleftrightarrow[2][]{\ext@arrow 0099{\longleftrightarrowfill@}{#1}{#2}}
\def\longleftrightarrowfill@{\arrowfill@\leftarrow\relbar\rightarrow}
\makeatother

\author{Stefano Riolo}
\address{Dipartimento di Matematica, Università di Bologna \newline  Piazza di Porta San Donato 5, 40126 Bologna, Italy}
\email{stefano\w{0}.\w{0}riolo\w{0}@unibo.it}

\title{A small cusped hyperbolic 4-manifold}

\subjclass[2010]{57M50}

\thanks{The author has been partially supported by 
the SNSF ``Ambizione'' grant PZ00P2-193559.} 

\begin{document}

\begin{abstract}
%
By gluing some copies of a polytope of Kerckhoff and Storm's, we build the smallest known orientable hyperbolic 4-manifold that is not commensurable with the ideal 24-cell or the ideal rectified simplex. It is cusped and arithmetic, and has twice the minimal volume. 
%
\end{abstract}

\maketitle

\section*{Introduction}

There is a natural interest in hyperbolic manifolds of low volume, and this note addresses dimension four; see the survey \cite{M_survey}.

Hyperbolic manifolds in the paper are understood to be complete and of finite volume. 
Two manifolds (or orbifolds) are \emph{commensurable} if they are both finitely covered by a third manifold (or orbifold). Recall also that the generalised Gauss--Bonnet formula relates the volume of a hyperbolic 4-manifold $M$ to its Euler characteristic $\chi(M)$ as follows:
$$\Vol(M) = \frac{4\pi^2}{3} \chi(M).$$
%

The smallest known closed hyperbolic 4-manifolds are non-orientable and have $\chi = 8$ \cite{CM,L}. These and the few other explicit examples of closed manifolds 
that we could find in the literature \cite{D,GS,MZ} are tessellated by $2\pi/k$-angled 120-cells for $k = 3, 4, 5$, and are arithmetic and commensurable \cite{JKRT}. 

Concerning cusped manifolds, we found instead four commensurability classes: (we cite here the papers containing the smallest known manifolds in each class)

\begin{enumerate}
\item \label{it:1} $\chi = 1$ both orientable and not, arithmetic \cite{RT,E,S_low,RT_1cusp,KRT}; 
\item \label{it:2} $\chi = 1$ both orientable and not, arithmetic \cite{
S_fig8,RS,KRR,CK}; 
\item \label{it:3} $\chi = 2$ non-orientable 
(so also $\chi = 4$ orientable), arithmetic \cite{MR,RS}; 
\item \label{it:4} $\chi = 3$ non-orientable and $\chi = 5$ orientable, non-arithmetic \cite{RS}.
\end{enumerate}
Each of these classes is represented by a hyperbolic Coxeter 4-polytope, for instance \eqref{it:1} by 
the nine non-compact simplices, $\matH^4/\mathrm{PO}(4,1;\mathbb Z)$ and the ideal 24-cell
, and \eqref{it:2} by the ideal rectified simplex
.

We find here another commensurability class of low-volume hyperbolic 4-manifolds:

\begin{thm*} \label{thm:main}
There exists an orientable, cusped, arithmetic, hyperbolic $4$-mani\-fold $M$ with $\chi(M) = 2$, 
not belonging to any of the commensurability classes 
above.
\end{thm*}

The manifold $M$ is commensurable with a Coxeter polytope $Q$ belonging to a continuous family of hyperbolic 4-polytopes discovered in 2010 by Kerckhoff and Storm \cite{KS}. Notably, classes \eqref{it:1} and \eqref{it:3} are represented by other Coxeter polytopes of the family, and the examples in \eqref{it:4} are hybrids of manifolds in \eqref{it:3} and \eqref{it:2}.
Our method applies to two additional Coxeter polytopes of the family, but giving less interesting examples from the viewpoint of this paper (see Section \ref{sec:same-method}).


We build $M$ explicitly, by gluing together some copies of a bigger polytope $P$ of Kerckhoff and Storm's (
$Q$ is a quotient of $P$). 
The results in \cite{MR,RS,RSep} are obtained in the same spirit, by gluing polytopes of the Kerckhoff--Storm family as well.

The polytope $P$ has volume $(2/5) \cdot (4\pi^2/3)$ and octahedral symmetry. It has a few 
2-faces with dihedral angle $2\pi/3$, while the remaining 
ones are right angled. These are the main features of $P$ that will be exploited.

As a first step, we ``kill'' in a 
natural way the $2\pi/3$ angles by gluing in pairs some facets 
of 5 copies of $P$, to get a particularly symmetric hyperbolic manifold $X$ with right-angled corners. These objects have been fruitfully used in four-dimensional hyperbolic geometry in the very last years \cite{M,MRS_1,MRS_2,BFS}.

A study of $X$ and its symmetries 
then allows us to build $M$ as follows. The facets of $X$ are  hyperbolic 3-manifolds with geodesic boundary (the corners), and are divided in two types. The corners, some punctured surfaces, are always the intersection of two facets of different type. We then close $X$ up via two commuting, fixed-point free, isometric involutions of $\partial X$ as gluing maps: one for each type of facet.

By construction, $M$ is commensurable with the orbifold $Q$. Being $Q$ arithmetic, Maclachlan's work \cite{Mac} allows us to distinguish its commensurability class.
One similarly gets a few non-orientable manifolds with $\chi = 2$ in the same class (Remark \ref{rem:non-or}).

\subsection*{Structure of the paper}

The polytopes $Q$ and $P$
are introduced in Section \ref{sec:polytopes}, while the 
manifolds $X$ and $M$ 
are built in Section \ref{sec:manifolds}. 

\subsection*{Acknowledgements}

I am grateful to Leone Slavich for several discussions on the topic. 

\section{The polytopes} \label{sec:polytopes}

We introduce here two polytopes from \cite{KS}: the Coxeter polytope $Q$ in Section \ref{sec:Q}, and the bigger polytope $P$ in Section \ref{sec:P}. The commensurability class of $Q$ is distinguished in Proposition \ref{prop:commens}.

We refer the reader to the book \cite{V} for the general theory of hyperbolic Coxeter polytopes and reflection groups, including arithmeticity.
%

\subsection{The Coxeter polytope} \label{sec:Q}

\begin{figure}
\begin{center}
\begin{tikzpicture}[-,>=stealth',shorten >=1pt,auto,node distance=1.3cm,
main node/.style={circle,draw,inner sep=1.5pt, outer sep=0pt,font=\sffamily\footnotesize\bfseries}]

\node[main node] (1) {$i_0$};
\node[main node] [right of=1] (2) {$\, c_{}^{\,}$};
\node[main node] [above of=1] (3) {$i_1$};
\node [right of=3] (9) {};
\node[main node] [right of=9] (4) {$\ell_{}^{\,}$};
\node[main node] [left of=3] (5) {$\, t_{}^{\,}$};
\node[main node] [below of=1] (6) {$i_2$};
\node [right of=6] (10) {};
\node[main node] [right of=10] (7) {$u_{}^{\,}$};
\node[main node] [left of=6] (8) {$b_{}^{\,}$};
 
\path[every node/.style={font=\sffamily\tiny}]

(1) edge [very thick] node {} (2)
(5) edge [very thick] node {} (8)
(5) edge node {} (3)
(8) edge node {} (6)
(3) edge [dashed] node {} (4)
(6) edge [dashed] node {} (7)
(4) edge [very thick] node {} (7)
(1) edge node {} (3)
(1) edge node {} (6);
\end{tikzpicture}
\end{center}
\caption{\footnotesize The Coxeter diagram of the reflection group $\Gamma$ of $Q$. If two nodes are joined by a thin, thick, or dashed edge, then the two corresponding bounding hyperplanes 
meet with angle $\pi/3$, are tangent at infinity, are ultraparallel, respectively. There is no edge joining two nodes if the corresponding hyperplanes are orthogonal.}\label{fig:coxeter_Q}
\end{figure}
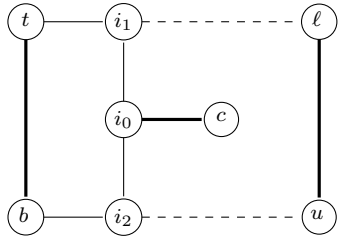

The Coxeter diagram in Figure \ref{fig:coxeter_Q} represents a non-compact hyperbolic Coxeter 4-polytope $Q \subset \matH^4$ of finite volume \cite[Proposition 13.1]{KS}, called $Q_{t_3}$ in \cite{KS}.

The associated reflection group
$$\Gamma = \langle \, t, \, b, \, u, \, \ell, \, c, \, i_0, \, i_1, \, i_2 \, \rangle < \Isom(\matH^4)$$
is arithmetic \cite[Theorem 13.2]{KS}. 

The symbols $t$, $b$, $u$, $\ell$, $c$, and $i$ are chosen to remind the words ``top'', ``bottom'', ``uper'', ``lower'', ``central'', and ``internal'', respectively. This may help later. 

Recall that the path graph with $n-1$ vertices is a Coxeter diagram for the symmetric group $\Sym_n$ 
\cite[Table 1]{V}. In the following sections, we will be interested in the edge $I$ of $Q$ corresponding to the reflection subgroup \begin{equation} \label{eq:G_I}
G_I = \langle i_0, i_1, i_2 \rangle \cong \Sym_4,
\end{equation}
and its vertices $V,V' \in I$, called \emph{top} and \emph{bottom} vertices, corresponding to the 
reflection groups
\begin{equation} \label{eq:G_V}
G_V = \langle i_0, i_1, i_2, t \rangle \cong \Sym_5, \quad G_{V'} = \langle i_0, i_1, i_2, b \rangle \cong \Sym_5.
\end{equation}

We conclude the section distinguishing the commensurability class of $Q$. 

\begin{prop} \label{prop:commens}
The arithmetic orbifold $Q = \matH^4 / \Gamma$ does not belong to any of the commensurability classes \eqref{it:1}, \eqref{it:2}, \eqref{it:3}, \eqref{it:4} mentioned in the introduction.
\end{prop}

\begin{proof}
By Maclachlan's work \cite{Mac}, the ramification set of the quadratic form associated to $\Gamma$ is a commensurability invariant of arithmetic reflection groups. We compute this set as explained in \cite{GJK}. For a quick description of the computation without proofs, one may consult \cite[Section 4.5]{MR}.

Up to isometry of $\matH^4$ in its hyperboloid model, the bounding hyperplanes of the polytope $Q \subset \matH^4$ (coherently oriented) are dually represented by these spacelike vectors of $\matR^{1,4}$:
\begin{eqnarray*}
t = \big( \sqrt{2}, 1, 1, 1, \sqrt7 \, \big), & & b = \big( \sqrt{2}, 1, 1, -1, -\sqrt7 \, \big),\\
u = \big( \sqrt{2}, 1, 1, 1, -\sqrt7/7 \, \big), & & \ell = \big( \sqrt{2}, 1, 1, -1, \sqrt7/7 \, \big),\\
c = \big( 1, \sqrt{2}, 0, 0, 0 \big), & & i_0 = \big( 0, -1, 1, 0, 0 \big),\\
i_1 = \big( 0, 0, -1, -1, 0 \big), & & i_2 = \big( 0, 0, -1, 1, 0 \big)
\end{eqnarray*}
(see \cite[Section 4]{KS} for $i_0$, $i_1$, $i_2$, and \cite[Table 3]{KS} with $t = t_3 = \sqrt7/7$ for the remaining vectors). The Gram matrix of the corresponding vectors of unit Minkowski norm is: 
$$\left(\begin{array}{rrrrrrrr}
1 & -1 & 0 & 0 & 0 & 0 & 0 & -\frac{1}{2} \\
-1 & 1 & 0 & 0 & 0 & 0 & -\frac{1}{2} & 0 \\
0 & 0 & 1 & -1 & 0 & 0 & 0 & -\frac{\sqrt7}{2} \\
0 & 0 & -1 & 1 & 0 & 0 & -\frac{\sqrt7}{2} & 0 \\
0 & 0 & 0 & 0 & 1 & -1 & 0 & 0 \\
0 & 0 & 0 & 0 & -1 & 1 & -\frac{1}{2} &
-\frac{1}{2} \\
0 & -\frac{1}{2} & 0 & -\frac{\sqrt7}{2} & 0
& -\frac{1}{2} & 1 & 0 \\
-\frac{1}{2} & 0 & -\frac{\sqrt7}{2} & 0 & 0
& -\frac{1}{2} & 0 & 1
\end{array}\right).$$

For some basis, a matrix representing the associated quadratic form over $\matQ$ is:
$$\left(\begin{array}{rrrrr}
7 & 0 & 0 & \frac{7}{2} & 0 \\
0 & 1 & -1 & 0 & 0 \\
0 & -1 & 1 & -\frac{1}{2} & -\frac{1}{2} \\
\frac{7}{2} & 0 & -\frac{1}{2} & 1 & 0 \\
0 & 0 & -\frac{1}{2} & 0 & 1
\end{array}\right).$$
Diagonalising the form, we get the matrix $\diag \left( 7/2,\, 1/2,\, 1/6,\, -3/8,\, 1/8 \right)$.  
From this datum, with the help of \cite[Propositions 4.8, 4.13, 4.15]{GJK}, it is straightforward to compute the ramification set of the form (as done in \cite[Proposition 4.25]{MR}), which is $\{ 2, 7 \}$.

The ramification sets of the arithmetic commensurability classes \eqref{it:1}, \eqref{it:2} and \eqref{it:3} are instead $\emptyset$, $\{ 3, \infty \}$ and $\{ 2, 5 \}$, respectively (see \cite[Proposition 2.5]{RS} and the references therein), while \eqref{it:4} is non-arithmetic. So the five classes are distinct.
\end{proof}

\subsection{The bigger polytope} \label{sec:P}

Recall Formula \eqref{eq:G_I}. By reflecting the Coxeter polytope $Q$ around its edge $I$, 
we get a bigger polytope
$$P = \bigcup_{g \in G_I} g(Q) \subset \matH^4$$
tessellated by $4!=24$ copies of $Q \cong P/G_I$. 
%

We resume here all the needed information on $P$
, referring to \cite{KS,MR} for proofs and further details (see especially \cite[Proposition 3.16]{MR}).

A remarkable property of $P$ is its antipodal symmetry. We call \emph{antipodal map} the inversion $a \in \Isom(P)$ through the centre of $P$ (the midpoint of the edge $I$ of $Q$). 

The polytope $P$ has 
22 facets, partitioned up to symmetry of $P$ into 3 sets:\footnote{In \cite{KS,MR} these are called: the ``odd'' and ``even'' ``positive walls'', the ``even'' and ``odd'' ``negative walls'', and the ``letter walls'', respectively.} 
\begin{itemize}
\item[(E)] the \emph{extremal facets}, divided into \emph{top} and \emph{bottom} facets
$$E_1, E_2, E_3, E_4\ \mbox{and}\ \, E'_1, E'_2, E'_3, E'_4,$$
tessellated by the copies of the facet $t$, resp. $b$, of $Q$, where $E'_i = a(E_i)$;
\item[(H)] the \emph{half-height facets}, divided into \emph{upper} and \emph{lower} facets
$$H_1, H_2, H_3, H_4\ \mbox{and}\ \, H'_1, H'_2, H'_3, H'_4,$$
tessellated by the copies of the facet $u$, resp. $\ell$, of $Q$, where $H'_i = a(H_i)$;
\item[(C)] the \emph{central facets}
$$C_{12}, C_{13}, C_{14}, C_{23}, C_{24}, C_{34},$$
tessellated by the copies of the facet $e$ of $Q$, where $C_{ij}$ is the unique such facet of $P$ that intersects the top facets $E_i$ and $E_j$. 
\end{itemize}

\begin{figure}
\begin{tikzpicture}[scale=.773,line cap=round,line join=round,>=triangle 45,x=1.0cm,y=1.0cm]
\clip(-0.18,-0.21) rectangle (18.74,5.49);
\fill[line width=0pt,color=red,fill=red,fill opacity=1.0] (2.47,2.9) -- (1.72,1.61) -- (2.26,1.3) -- (2.26,0.69) -- (3.74,0.69) -- (3.74,1.3) -- (4.28,1.61) -- (3.53,2.9) -- (3,2.59) -- cycle;
\draw [line width=1pt] (0,0)-- (1.72,1.61);
\draw [line width=1pt] (1.72,1.61)-- (2.26,1.3);
\draw [line width=1pt] (0,0)-- (2.26,0.69);
\draw [line width=1pt] (2.26,0.69)-- (2.26,1.3);
\draw [line width=1pt] (3,5.2)-- (2.47,2.9);
\draw [line width=1pt] (2.47,2.9)-- (3,2.59);
\draw [line width=1pt] (3,5.2)-- (3.53,2.9);
\draw [line width=1pt] (3.53,2.9)-- (3,2.59);
\draw [line width=1pt] (6,0)-- (4.28,1.61);
\draw [line width=1pt] (4.28,1.61)-- (3.74,1.3);
\draw [line width=1pt] (3.74,0.69)-- (3.74,1.3);
\draw [line width=1pt] (6,0)-- (3.74,0.69);
\draw [line width=1pt] (1.72,1.61)-- (2.47,2.9);
\draw [line width=1pt] (3.53,2.9)-- (4.28,1.61);
\draw [line width=1pt] (3.74,0.69)-- (2.26,0.69);
\draw [line width=1pt,color=yellow] (2.26,1.3)-- (3,1.73);
\draw [line width=1pt,color=yellow] (3,1.73)-- (3.74,1.3);
\draw [line width=1pt,color=yellow] (3,1.73)-- (3,2.59);
\draw [line width=1pt] (10.09,5.2)-- (10.81,2.15);
\draw [line width=1pt] (10.81,2.15)-- (13.09,0);
\draw [line width=1pt] (10.09,5.2)-- (9.37,2.15);
\draw [line width=1pt] (7.09,0)-- (10.09,0.91);
\draw [line width=1pt] (10.09,0.91)-- (13.09,0);
\draw [line width=1pt] (7.09,0)-- (9.37,2.15);
\draw [line width=1pt,color=red] (9.37,2.15)-- (10.09,1.73);
\draw [line width=1pt,color=red] (10.09,1.73)-- (10.81,2.15);
\draw [line width=1pt,color=red] (10.09,1.73)-- (10.09,0.91);
\draw [line width=1pt] (7.09,0)-- (10.09,5.2);
\draw [line width=1pt] (10.09,5.2)-- (13.09,0);
\draw [line width=1pt] (13.09,0)-- (7.09,0);
\draw [line width=1pt] (6,0)-- (3,5.2);
\draw [line width=1pt] (3,5.2)-- (0,0);
\draw [line width=1pt] (0,0)-- (6,0);
\draw [line width=1pt] (13.97,4.92)-- (18.31,4.92);
\draw [line width=1pt] (18.31,4.92)-- (18.31,0.58);
\draw [line width=1pt] (18.31,0.58)-- (13.97,0.58);
\draw [line width=1pt] (13.97,0.58)-- (13.97,4.92);
\draw [line width=1pt,dash pattern=on 2pt off 4pt] (13.97,4.92)-- (16.14,3.82);
\draw [line width=1pt,dash pattern=on 2pt off 4pt] (16.14,3.82)-- (18.31,4.92);
\draw [line width=1pt,dash pattern=on 2pt off 4pt] (13.97,0.58)-- (16.14,1.68);
\draw [line width=1pt,dash pattern=on 2pt off 4pt] (16.14,1.68)-- (18.31,0.58);
\draw [line width=1pt,dash pattern=on 2pt off 4pt,color=red] (16.14,3.82)-- (16.14,1.68);
\draw [line width=1pt,color=red] (15.07,2.75)-- (17.22,2.75);
\draw [line width=1pt] (13.97,4.92)-- (15.07,2.75);
\draw [line width=1pt] (15.07,2.75)-- (13.97,0.58);
\draw [line width=1pt] (17.22,2.75)-- (18.31,4.92);
\draw [line width=1pt] (17.22,2.75)-- (18.31,0.58);
\draw (0.7,5.55) node[anchor=north west] {$ E_4 $};
\draw (2.6,1.54) node[anchor=north west] {\footnotesize $E_2$};
\draw (2.1,2.44) node[anchor=north west] {\footnotesize $E_1$};
\draw (3.1,2.44) node[anchor=north west] {\footnotesize $E_3 $};
\draw (2.6,0.72) node[anchor=north west] {\footnotesize $C_{24} $};
\draw (3.7,2.8) node[anchor=north west] {\footnotesize $C_{34} $};
\draw (1.35,2.8) node[anchor=north west] {\footnotesize $C_{14} $};
\draw (1.4,1.4) node[anchor=north west] {\footnotesize $H_3 $};
\draw (3.9,1.4) node[anchor=north west] {\footnotesize $H_1 $};
\draw (2.6,3.73) node[anchor=north west] {\footnotesize $H_2 $};
\draw (3.9,4.1) node[anchor=north west] {\footnotesize $H'_4 $};
\draw (7.8,5.55) node[anchor=north west] {$ H_4 $};
\draw (9.7,2.69) node[anchor=north west] {\footnotesize $E_2 $};
\draw (9.05,1.7) node[anchor=north west] {\footnotesize $E_3 $};
\draw (10.97,4.1) node[anchor=north west] {\footnotesize $E'_4 $};
\draw (8.5,2.8) node[anchor=north west] {\footnotesize $C_{23} $};
\draw (10.8,2.8) node[anchor=north west] {\footnotesize $C_{12} $};
\draw (9.7,0.72) node[anchor=north west] {\footnotesize $C_{13} $};
\draw (10.3,1.7) node[anchor=north west] {\footnotesize $E_1 $};
\draw (12.7,5.55) node[anchor=north west] {$ C_{12} $};
\draw (16.2,2.7) node[anchor=north west] {\footnotesize $E_1 $};
\draw (15.25,3.45) node[anchor=north west] {\footnotesize $E_2 $};
\draw (17.35,3.1) node[anchor=north west] {\footnotesize $H_3 $};
\draw (14.1,3.1) node[anchor=north west] {\footnotesize $H_4 $};
\draw (15.75,5.55) node[anchor=north west] {\footnotesize $H'_2 $};
\draw (13.28,3.15) node[anchor=north west] {\footnotesize $E'_4$};
\draw (18.1615,3.15) node[anchor=north west] {\footnotesize $E'_3$};
\draw (15.75,0.69) node[anchor=north west] {\footnotesize $H'_1 $};
\begin{scriptsize}
\fill [color=blue] (3,1.73) circle (2pt);
\fill [color=black] (1.72,1.61) circle (2pt);
\fill [color=green] (2.26,1.3) circle (2pt);
\fill [color=black] (0,0) circle (3.5pt);
\fill [color=white] (0,0) circle (2.5pt);
\fill [color=black] (2.26,0.69) circle (2pt);
\fill [color=black] (2.47,2.9) circle (2pt);
\fill [color=black] (3,5.2) circle (3.5pt);
\fill [color=white] (3,5.2) circle (2.5pt);
\fill [color=black] (3.53,2.9) circle (2pt);
\fill [color=green] (3,2.59) circle (2pt);
\fill [color=black] (4.28,1.61) circle (2pt);
\fill [color=green] (3.74,1.3) circle (2pt);
\fill [color=black] (6,0) circle (3.5pt);
\fill [color=white] (6,0) circle (2.5pt);
\fill [color=black] (3.74,0.69) circle (2pt);
\fill [color=green] (10.09,1.73) circle (2pt);
\fill [color=black] (10.81,2.15) circle (2pt);
\fill [color=black] (10.09,5.2) circle (3.5pt);
\fill [color=white] (10.09,5.2) circle (2.5pt);
\fill [color=black] (10.09,0.91) circle (2pt);
\fill [color=black] (13.09,0) circle (3.5pt);
\fill [color=white] (13.09,0) circle (2.5pt);
\fill [color=black] (7.09,0) circle (3.5pt);
\fill [color=white] (7.09,0) circle (2.5pt);
\fill [color=black] (9.37,2.15) circle (2pt);
\fill [color=black] (13.97,4.92) circle (3.5pt);
\fill [color=white] (13.97,4.92) circle (2.5pt);
\fill [color=black] (18.31,4.92) circle (3.5pt);
\fill [color=white] (18.31,4.92) circle (2.5pt);
\fill [color=black] (13.97,0.58) circle (3.5pt);
\fill [color=white] (13.97,0.58) circle (2.5pt);
\fill [color=black] (18.31,0.58) circle (3.5pt);
\fill [color=white] (18.31,0.58) circle (2.5pt);
\fill [color=black] (16.14,3.82) circle (2pt);
\fill [color=black] (16.14,1.68) circle (2pt);
\fill [color=black] (15.07,2.75) circle (2pt);
\fill [color=black] (17.22,2.75) circle (2pt);
\end{scriptsize}
\end{tikzpicture}
\caption{\footnotesize 
The top, upper, and central facets $E_4$, $H_4$, and $C_{12}$ of $P$. The ideal vertices are drawn as white dots, while the top vertex $V$ and the type-2 and type-3 vertices are in blue, green and black, respectively. The black, red, and yellow edges have dihedral angle $\pi/2$, $2\pi/3$, and $\mathrm{arcos}(-1/3)$, respectively. The red pentagons are $2\pi/3$-angled ridges of $P$, while the other 
2-faces in the picture are right angled.}\label{fig:facets}
\end{figure}
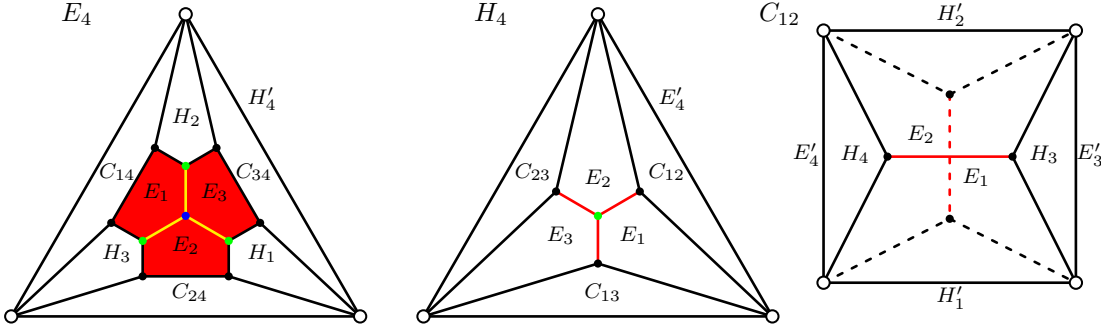

The ideal vertices of $P$ all lie in the ideal boundary of a hyperplane $H^3 \subset \matH^4$ 
\cite[Proposition 3.19]{MR}. By definition, the names of the facets agree on whether they are contained in one of the two half-spaces bounded by $H^3$ (``top/bottom'' and ``upper/lower'') or not (``central''). The group $G_I$ preserves the two half-spaces, while $a$ exchanges them.

With this conventions, all is well defined up to co-orientation of $H^3$ (``top/bottom'' and ``upper/lower'') and permutation of the facet indices, thanks to the following lemma.
Let $\Alt_n < \Sym_n$ denote the alternating subgroup. 

\begin{lemma} \label{lem:symmetry_P}
The action of $\Sym_4$ on the facet 
indices
is induced by an isomorphism
$$\Isom(P) \cong \matZ/2\matZ \times \Sym_4$$
which restricts to 
$$\langle a \rangle \cong \matZ/2\matZ \times \{ \id \}, \quad G_I \cong \{0\} \times \Sym_4, \quad\mbox{and}\quad \Isom^+(P) \cong \matZ/2\matZ \times \Alt_4.$$
%
Moreover, the quotient $P/\Isom(P)$ is isometric to the orbifold $Q/\Isom(Q)$.
\end{lemma}

\begin{proof}
The same argument of \cite[Proposition 2.4]{RS} applies, replacing the words ``upper tetrahedral facet'' with ``link of the top vertex 
$V$'' 
(see also \cite[Lemma 4.15]{MR}). The key point is that the unique non-trivial symmetry of $Q$ (an order-two rotation corresponding to the reflection along the edge $\{ i_0, c \}$ of the diagram in Figure \ref{fig:coxeter_Q}) writes as the composition of $a$ with an element of $G_I$. 
\end{proof}

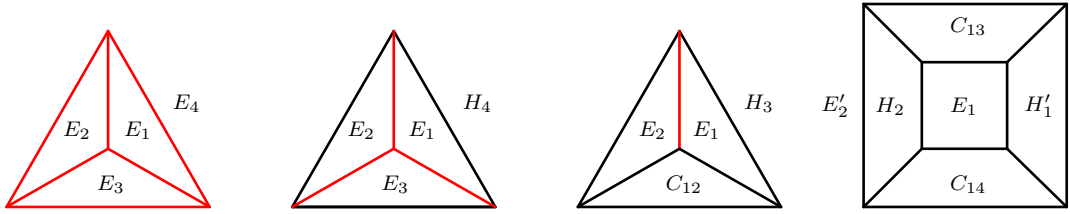
\begin{figure}
\begin{tikzpicture}[scale=1.8,line cap=round,line join=round,>=triangle 45,x=1.0cm,y=1.0cm]
\clip(-0.07,-0.06) rectangle (7.92,1.6);
\draw [line width=1pt] (6.76,1.07)-- (7.39,1.07);
\draw [line width=1pt] (7.39,1.07)-- (7.39,0.43);
\draw [line width=1pt] (7.39,0.43)-- (6.76,0.43);
\draw [line width=1pt] (6.76,0.43)-- (6.76,1.07);
\draw [line width=1pt] (4.22,0)-- (5.72,0);
\draw [line width=1pt] (5.72,0)-- (4.97,1.3);
\draw [line width=1pt] (4.97,1.3)-- (4.22,0);
\draw [line width=1pt] (3.61,0)-- (2.86,1.3);
\draw [line width=1pt] (2.86,1.3)-- (2.11,0);
\draw [line width=1pt] (2.11,0)-- (3.61,0);
\draw [line width=1pt] (2.11,0)-- (3.61,0);
\draw [color=red,line width=1pt] (1.5,0)-- (0.75,1.3);
\draw [color=red,line width=1pt] (0.75,1.3)-- (0,0);
\draw [color=red,line width=1pt] (0,0)-- (1.5,0);
\draw [line width=1pt] (6.33,1.5)-- (6.33,0);
\draw [line width=1pt] (6.33,0)-- (7.83,0);
\draw [line width=1pt] (7.83,0)-- (7.83,1.5);
\draw [line width=1pt] (7.83,1.5)-- (6.33,1.5);
\draw [line width=1pt] (6.33,1.5)-- (6.76,1.07);
\draw [line width=1pt] (6.76,0.43)-- (6.33,0);
\draw [line width=1pt] (7.39,0.43)-- (7.83,0);
\draw [line width=1pt] (7.39,1.07)-- (7.83,1.5);
\draw [color=red,line width=1pt] (0,0)-- (0.75,0.43);
\draw [color=red,line width=1pt] (0.75,1.3)-- (0.75,0.43);
\draw [color=red,line width=1pt] (0.75,0.43)-- (1.5,0);
\draw [color=red,line width=1pt] (2.11,0)-- (2.86,0.43);
\draw [color=red,line width=1pt] (2.86,0.43)-- (2.86,1.3);
\draw [color=red,line width=1pt] (2.86,0.43)-- (3.61,0);
\draw [line width=1pt] (4.22,0)-- (4.97,0.43);
\draw [color=red,line width=1pt] (4.97,0.43)-- (4.97,1.3);
\draw [line width=1pt] (4.97,0.43)-- (5.72,0);
\draw (0.6,0.3) node[anchor=north west] {\footnotesize $E_3$};
\draw (0.8,0.7) node[anchor=north west] {\footnotesize $E_1$};
\draw (0.35,0.7) node[anchor=north west] {\footnotesize $E_2$};
\draw (1.16,0.9) node[anchor=north west] {\footnotesize $E_4$};
\draw (2.7,0.3) node[anchor=north west] {\footnotesize $E_3$};
\draw (2.9,0.7) node[anchor=north west] {\footnotesize $E_1$};
\draw (2.45,0.7) node[anchor=north west] {\footnotesize $E_2$};
\draw (3.3,0.9) node[anchor=north west] {\footnotesize $H_4$};
\draw (5.0,0.7) node[anchor=north west] {\footnotesize $E_1$};
\draw (4.6,0.7) node[anchor=north west] {\footnotesize $E_2$};
\draw (4.8,0.3) node[anchor=north west] {\footnotesize $C_{12}$};
\draw (5.38,0.9) node[anchor=north west] {\footnotesize $H_3$};
\draw (6.9,0.88) node[anchor=north west] {\footnotesize $E_1$};
\draw (5.95,0.9) node[anchor=north west] {\footnotesize $E'_2$};
\draw (6.9,1.47) node[anchor=north west] {\footnotesize $C_{13}$};
\draw (6.9,0.3) node[anchor=north west] {\footnotesize $C_{14}$};
\draw (6.35,0.88) node[anchor=north west] {\footnotesize $H_2$};
\draw (7.45,0.9) node[anchor=north west] {\footnotesize $H'_1$};
\end{tikzpicture}
\caption{\footnotesize The vertex links of $P$ up to symmetry: three spherical tetrahedra for the finite vertices (from left to right, of type 1, 2 and 3), and a Euclidean 
parallelepiped for the ideal vertices. 
The black edges are right angled, while the red ones are $2\pi/3$ angled.}\label{fig:vertices}
\end{figure}

Some facets of $P$ are drawn in Figure \ref{fig:facets}, and some vertex links in Figure \ref{fig:vertices}. From these two pictures and the respective captions, one recovers the whole combinatorics and geometry of $P$ (like for instance the adjacency graph of the facets) by means of Lemma \ref{lem:symmetry_P}. We record here a few consequences that will be relevant for us.

First, the top (resp. bottom) facets intersect each other with angle $2\pi/3$, 
and share the top (resp. bottom) vertex $V = E_1 \cap \ldots \cap E_4$ (resp. $V' = a(V) = E'_1 \cap \ldots \cap E'_4$). Second, the central, resp. half-height, facets are pairwise disjoint. 
Third, if two non-isometric facets intersect, then they are orthogonal. Fourth, the antipodal map acts on the central facets as follows:
\begin{equation} \label{eq:a}
a(C_{i,j}) = C_{k,l} \quad \mbox{for\ all\ distinct}\ i, j, k, l.
\end{equation}

%
%

%
%

\section{The construction} \label{sec:manifolds}

In this section we build the manifold $M$ by gluing some copies of $P$. Recall that the top and bottom vertices of $P$ have as link the spherical tetrahedron with dihedral angles $2\pi/3$. The latter tiles $S^3$ in the regular honeycomb combinatorially equivalent to the triangulation of the boundary of the 4-simplex. Therefore we need at least 5 copies of $P$ to build a manifold, and we will see that 5 is enough.

We proceed as follows. In Section \ref{sec:step1}, we build a manifold with corners $X$ by gluing in pairs the extremal facets of 5 copies of $P$. In Section \ref{sec:X}, we study $X$ and its symmetries. In Section \ref{sec:step2}, we close $X$ up by pairing the remaining facets 
via two symmetries of $X$, and conclude our proof. Finally, in Section \ref{sec:same-method}, we give some additional information on how to apply the construction to two more polytopes of Kerckhoff and Storm's.

Before beginning with the construction, let us introduce here some terminology.
%
%
A \emph{hyperbolic} \emph{$n$-manifold with ``pure'' right-angled corners} is a 
hyperbolic manifold $X$ with boundary, locally modelled on the intersection in $\matH^n$ of two closed half-spaces bounded by orthogonal hyperplanes. The boundary $\partial X$ is naturally stratified into maximal connected, totally geodesic, submanifolds: the $(n-1)$-dimensional \emph{facets} have (possibly empty) totally geodesic boundary, while the $(n-2)$-dimensional \emph{corners} have no boundary.
We say that $\partial X$ is \emph{bicolourable} if there are two unions of facets $A$ and $B$ such that $\partial X = A \cup B$ and the corners of $X$ are the connected components of $A \cap B$. This implies that each corner is contained in exactly two facets: one in $A$, and one in $B$. In particular, the facets of $X$ are isometrically embedded hyperbolic manifolds with totally geodesic boundary (made of the corners of $X$).

\subsection{Extremal gluing} \label{sec:step1}

We build here 
$X$ 
by gluing isometrically in pairs all the extremal facets of 5 copies of $P$: the top ones first and the bottom ones after. There is in fact a unique way to perform the first gluing, and a preferred one for the second. 

Recall Formula \eqref{eq:G_V}.
By reflecting the smaller Coxeter polytope $Q \subset P$ around its top vertex $V = E_1 \cap \ldots \cap E_4$,
we get a big polytope
$$\tilde P = \bigcup_{g \in G_V} g(Q) \subset \matH^4.$$
Since $G_I < G_V$ with index $5 = 5! / 4!$, the polytope $\tilde P$ contains $P$ and is tesselleted by 5 isometric copies of $P$
. These are glued altogether around $V$ (the centre of $\tilde P$), in the pattern of the triangulation of $S^3$ induced by the boundary of the 4-simplex. 

\begin{figure}
\begin{center}
\begin{tikzpicture}[-,>=stealth' ,shorten >=1pt,auto,node distance=2cm,
main node/.style={circle,draw,inner sep=2pt, outer sep=0pt,font=\sffamily\bfseries}]

\node[main node] (2) {$2$};
\node [right of=2] (5) {};
\node[main node] [right of=5] (1) {$1$};
\node [below of=2] (6) {};
\node [right of=6] (7) {};
\node [right of=7] (8) {};
\node[main node] [right of=8] (0) {$0$};
\node[main node] [below of=6] (3) {$3$};
\node [right of=3] (9) {};
\node[main node] [right of=9] (4) {$4$};
 
\path[every node/.style={font=\sffamily\tiny}]

(0) edge [blue] node [near end, above] {$1$} (1)
(0) edge [blue] node [near start, above] {$1$} (1)
(0) edge [blue] node [near end, below] {$2$} (2)
(0) edge [blue] node [near start, below] {$2$} (2)
(0) edge [blue] node [near end, above] {$3$} (3)
(0) edge [blue] node [near start, above] {$3$} (3)
(0) edge [blue] node [near end, below] {$4$} (4)
(0) edge [blue] node [near start, below] {$4$} (4)
(1) edge node [near end, above] {$1$} (2)
(1) edge node [near start, above] {$2$} (2)
(1) edge node [near end, above] {$1$} (3)
(1) edge node [near start, above=5pt] {$3$} (3)
(1) edge node [near end, right] {$1$} (4)
(1) edge node [near start, right] {$4$} (4)
(2) edge node [near end, left] {$2$} (3)
(2) edge node [near start, left] {$3$} (3)
(2) edge node [near end, left] {$2$} (4)
(2) edge node [near start, left] {$4$} (4)
(3) edge node [near end, above] {$3$} (4)
(3) edge node [near start, above] {$4$} (4);
\end{tikzpicture}
\end{center}
\caption{\footnotesize The complete graph $K_5$, with its nodes and half-edges labelled to remind the ``more abstract'' construction of $X$ and $\tilde{P}$, which follows here. To build $X$ (resp. $\tilde{P}$), we glue the extremal (resp. top) facets of 5 abstract copies $P_0, \ldots, P_4$ of $P$ as follows. Consider the edge 
of $K_5$ joining the nodes $i$ and $j$, where $i < j$. If $i = 0$ (blue edge), glue the facets $E_j$  and $E'_j$ (resp. the facet $E_j$) of $P_0$ to the corresponding facets of $P_j$ via the map induced by the identity of $P$. If instead $i \neq 0$ (black edge), glue the facets $E_j$ and $E'_j$ (resp. the facet $E_j$) of $P_i$ to the facets $E_i$ and $E'_i$ (resp. the facet $E_i$) of $P_j$ via the map induced by the reflection $(ij) \in \mathfrak{S}_4 = G_I
$ of $P$.}\label{fig:K5}
\end{figure}
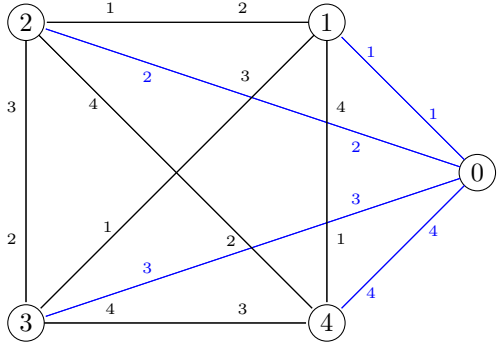

Note that we can alternatively think of $\tilde P$ as built by pairing isometrically all the top facets of 5 abstract copies of $P$ as described in Figure \ref{fig:K5}. We now build $X$ from $\tilde P$ by pairing all the bottom facets of these 5 copies of $P$ as follows.
Let $\varphi \colon F \to G$ be a pairing map between two top facets of two copies $P_x$ and $P_y$ of $P$
. Then we 
want to glue
$a_x(F)$ to $a_y(G)$ via $a_y \circ \varphi \circ a_x$, where 
$a_z$ denotes the antipodal map of $P_z$. This is done in Figure \ref{fig:K5}.

More concretely, let $r_i$ be the reflection along the top face $E_i$ of $P$, and define $P_i = r_i(P) \subset \matH^4$ (so that $a_i = r_i \circ a \circ r_i$). 
To build $X$ from
$$\tilde P = P \cup P_1 \cup \ldots \cup P_4 \subset \matH^4,$$
we glue the facet $E'_i$ of $P$ to the facet $r_i(E'_i)$ of $P_i$ via $r_i$, and, for $i \neq j$, the facet $r_i(E'_j)$ of $P_i$ to the facet $r_j(E'_i)$ of $P_j$ via 
the reflection $r_{ij} \in G_I$ corresponding to $(ij) \in \Sym_4$ 
(recall Lemma \ref{lem:symmetry_P} and see Figure \ref{fig:construction}).

\begin{figure}
\begin{tikzpicture}[scale=.15,line cap=round,line join=round,>=triangle 45,x=1.0cm,y=1.0cm]
\draw [shift={(-4.22,-8.21)},line width=1.2pt]  plot[domain=-0.52:1.57,variable=\t]({1*7.14*cos(\t r)+0*7.14*sin(\t r)},{0*7.14*cos(\t r)+1*7.14*sin(\t r)});
\draw [shift={(15,25.08)},line width=1.2pt]  plot[domain=3.67:5.76,variable=\t]({1*7.14*cos(\t r)+0*7.14*sin(\t r)},{0*7.14*cos(\t r)+1*7.14*sin(\t r)});
\draw [shift={(34.22,-8.21)},line width=1.2pt]  plot[domain=1.57:3.67,variable=\t]({1*7.14*cos(\t r)+0*7.14*sin(\t r)},{0*7.14*cos(\t r)+1*7.14*sin(\t r)});
\draw [line width=1.2pt] (15,17.95)-- (15,2.89);
\draw [line width=1.2pt] (15,2.89)-- (1.96,-4.64);
\draw [line width=1.2pt] (15,2.89)-- (28.04,-4.64);
\draw [line width=1.2pt] (-4.22,13.98)-- (-4.22,-1.07);
\draw [line width=1.2pt] (-4.22,6.08) -- (-4.69,6.45);
\draw [line width=1.2pt] (-4.22,6.08) -- (-3.75,6.45);
\draw [line width=1.2pt] (-4.22,6.83) -- (-4.69,7.21);
\draw [line width=1.2pt] (-4.22,6.83) -- (-3.75,7.21);
\draw [line width=1.2pt] (-4.22,5.32) -- (-4.69,5.7);
\draw [line width=1.2pt] (-4.22,5.32) -- (-3.75,5.7);
\draw [line width=1.2pt] (15,-19.31)-- (1.96,-11.78);
\draw [line width=1.2pt] (7.82,-15.16) -- (8.39,-14.94);
\draw [line width=1.2pt] (7.82,-15.16) -- (7.91,-15.76);
\draw [line width=1.2pt] (8.48,-15.54) -- (9.04,-15.32);
\draw [line width=1.2pt] (8.48,-15.54) -- (8.57,-16.14);
\draw [line width=1.2pt] (15,-19.31)-- (28.04,-11.78);
\draw [line width=1.2pt] (21.85,-15.35) -- (21.76,-15.95);
\draw [line width=1.2pt] (21.85,-15.35) -- (21.28,-15.13);
\draw [line width=1.2pt] (34.22,13.98)-- (21.18,21.51);
\draw [line width=1.2pt] (27.04,18.13) -- (27.61,18.35);
\draw [line width=1.2pt] (27.04,18.13) -- (27.14,17.53);
\draw [line width=1.2pt] (27.7,17.75) -- (28.27,17.97);
\draw [line width=1.2pt] (27.7,17.75) -- (27.79,17.15);
\draw [line width=1.2pt] (-4.22,13.98)-- (8.82,21.51);
\draw [line width=1.2pt] (2.63,17.94) -- (2.54,17.34);
\draw [line width=1.2pt] (2.63,17.94) -- (2.06,18.16);
\draw [line width=1.2pt] (34.22,13.98)-- (34.22,-1.07);
\draw [line width=1.2pt] (34.22,6.08) -- (33.75,6.45);
\draw [line width=1.2pt] (34.22,6.08) -- (34.69,6.45);
\draw [line width=1.2pt] (34.22,6.83) -- (33.75,7.21);
\draw [line width=1.2pt] (34.22,6.83) -- (34.69,7.21);
\draw [line width=1.2pt] (34.22,5.32) -- (33.75,5.7);
\draw [line width=1.2pt] (34.22,5.32) -- (34.69,5.7);
\draw [shift={(-4.75,-8.52)}] plot[domain=-0.22:1.27,variable=\t]({1*2.23*cos(\t r)+0*2.23*sin(\t r)},{0*2.23*cos(\t r)+1*2.23*sin(\t r)});
\draw [shift={(34.75,-8.52)}] plot[domain=-0.22:1.27,variable=\t]({-1*2.23*cos(\t r)+0*2.23*sin(\t r)},{0*2.23*cos(\t r)+1*2.23*sin(\t r)});
\draw [shift={(15,25.69)}] plot[domain=-0.22:1.27,variable=\t]({0.5*2.23*cos(\t r)+-0.87*2.23*sin(\t r)},{-0.87*2.23*cos(\t r)+-0.5*2.23*sin(\t r)});
\draw [dotted] (-4.56,-8.41)-- (36.37,15.22);
\draw [dotted] (15,25.47)-- (15,-21.79);
\draw [dotted] (34.56,-8.41)-- (-6.37,15.22);
\draw (-3.92,-6.93)-- (-4.09,-6.39);
\draw (-4.09,-6.39)-- (-3.66,-6.19);
\draw (33.92,-6.93)-- (34.09,-6.39);
\draw (34.09,-6.39)-- (33.66,-6.19);
\draw (14.04,24.18)-- (13.48,24.06);
\draw (13.48,24.06)-- (13.53,23.58);
\draw (14.5,-5.5) node[anchor=north west] {$P$};
\draw (-10,24) node[anchor=north west] {$\tilde P \subset \matH^4$};
\draw (14.5,2) node[anchor=north west] {\tiny $V$};
\draw (14.5,-19) node[anchor=north west] {\tiny $V'$};
\draw (19.5,10.55) node[anchor=north west] {$P_i$};
\draw (4,8) node[anchor=north west] {$P_j$};
\draw (-2.8,-7.8) node[anchor=north west] {\footnotesize $r_j$};
\draw (32.1,-7.8) node[anchor=north west] {\footnotesize $r_i$};
\draw (15.74,24.17) node[anchor=north west] {\footnotesize $r_{ij}$};
\draw (18.5,-0.35) node[anchor=north west] {\tiny $E_i$};
\draw (7.8,-0.35) node[anchor=north west] {\tiny $E_j$};
\draw (18.5,-12) node[anchor=north west] {\tiny $E'_j$};
\draw (7.8,-12) node[anchor=north west] {\tiny $E'_i$};
\fill [color=blue] (15,2.89) circle (13pt);
\fill [color=blue] (-4.22,13.98) circle (13pt);
\fill [color=blue] (15,-19.31) circle (13pt);
\fill [color=blue] (34.22,13.98) circle (13pt);
\end{tikzpicture}

\vspace{-.4cm}

\caption{\footnotesize A schematic picture of the ``more concrete'' construction of $X$ from $\tilde{P} = P \cup P_1 \cup \ldots \cup P_4 \subset \matH^4$.} \label{fig:construction}
\end{figure}
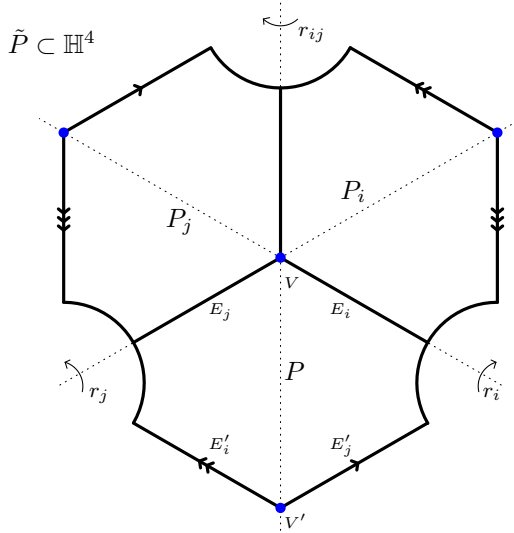

\subsection{The manifold with corners} \label{sec:X}

We study here the 
space $X$ just constructed.
%

Let $H$ and $C$ denote the union of the copies of the half-height and central facets of $P$ in $X$, respectively
. These facets are the unpaired ones.

\begin{prop} \label{prop:X}
The resulting complex $X$ is an orientable hyperbolic $4$-manifold with pure right-angled corners and boundary $\partial X = H \cup C$ bicoloured by $H$ and $C$.
\end{prop}

\begin{proof}
We analyse the effect of the gluing on the vertex links of the copies of $P$ along their extremal facets. 

We refer to Figures \ref{fig:vertices} and \ref{fig:K5}. By construction, the gluing graphs of the type-$k$ vertex links correspond to the subgraphs of $K_5$ spanned by $6-k$ nodes (repeated twice to deal with the antipodal vertices). Precisely, the type-1 links are glued so as to tessellate two copies of $S^3$ (thus giving points in the interior of $X$) like the boundary of the 4-simplex, the type-2 links tessellate some copies of a half-sphere of $S^3$ (giving internal points of the facets of $X$), and the type-3 links tessellate some copies of a right-angled bigon of $S^3$ (giving points of the corners of $X$). The links of the ideal vertices are glued to form flat 3-manifolds with right-angled corners, stratum-preserving homeomorphic to $[0,1]^2$-bundles over $S^1$
(giving the cusps of $X$).

We have shown that the vertex links glue to form spherical or flat 3-manifolds with pure right-angled corners, so $X$ is a hyperbolic 4-manifold with pure right-angled corners. It is orientable, being obtained from $\tilde P$ via orientation-reversing gluing maps.

By construction $\partial X = H \cup C$ and $H$ and $C$ are unions of facets. The previous analysis also shows that the corners are contained in the intersection of copies of a half-height and a central facet of $P$ (see the gluing of the type-3 vertex links). Therefore $H$ and $C$ define a bicolouring of $\partial X$, and this concludes the proof.
\end{proof}

It is not difficult to show that $H$ consists of two disjoint facets (an ``upper'' one and a ``lower'' one), while $C$ is a facet. We omit the proof since this fact is not needed.

Recall now Lemma \ref{lem:symmetry_P} and Formulas \eqref{eq:G_I} and \eqref{eq:G_V}. In the sequel, we shall think of $\Sym_4 \subset \Sym_5$ as the stabiliser of $0$ in the permutation group $\Sym_5$ of $\{ 0, \ldots, 4 \}$. 

\begin{lemma} \label{lem:sym_X}
Every symmetry of $P$ is the restriction of an isometry of $X$.
\end{lemma}

\begin{proof}
This is true for $a$ thanks for the ``$a$-equivariance'' of the gluing maps producing $X$, and for each $\sigma \in \Sym_4 \subset \Sym_5 \cong \Aut(K_5)$ since it induces a permutation ($\sigma$ itself!) of the half-edge labels of the graph $K_5$ in Figure \ref{fig:K5}. 
\end{proof}

Since moreover $G_V \cong \Sym_5$ permutes the 5 copies of $P$ in $X$
, we can write:
$$\Isom(P) = \matZ/2\matZ \times \Sym_4 \subset \matZ/2\matZ \times \Sym_5 \subset \Isom(X).$$
In the following section, we will implicitly use Lemma \ref{lem:sym_X} and adopt the above convention.

\subsection{Half-height and central gluing} \label{sec:step2}

We are finally ready to build $M$ from $X$.

Let $r_{ij} \in \Isom(X)$ 
correspond to the reflection $(ij) \in \Sym_4 \subset \Isom(P)$, and recall that $\partial X = H \cup C$
. Now pick $X$, glue $H$ to itself via $h = a \circ r_{12}$ and $C$ to itself via $c = a \circ r_{34}$, and call $M$ the resulting space. 

\begin{prop} \label{prop:M}
The resulting complex $M$ is an orientable hyperbolic $4$-manifold with $\chi(M) = 2$, commensurable with the orbifold $Q = \matH^4/\Gamma$.
\end{prop}

\begin{proof}
Observe that $h,c \in \Isom(P)$ are both the composition of an inversion through a point with a reflection through a hyperplane containing the point, so their fixed-point set is a line. In both cases, this line joins two ideal vertices of $P$:
$$\overline{E_1} \cap \overline{E'_2} \cap \overline{H_2} \cap \overline{H'_1} \cap \overline{C_{13}} \cap \overline{C_{14}} \quad \mbox{and} \quad \overline{E_2} \cap \overline{E'_1} \cap \overline{H_1} \cap \overline{H'_2} \cap \overline{C_{23}} \cap \overline{C_{24}},$$
$$\overline{E_3} \cap \overline{E'_4} \cap \overline{H_4} \cap \overline{H'_3} \cap \overline{C_{13}} \cap \overline{C_{23}} \quad \mbox{and} \quad \overline{E_4} \cap \overline{E'_3} \cap \overline{H_3} \cap \overline{H'_4} \cap \overline{C_{14}} \cap \overline{C_{24}},$$
respectively, as it is easily checked by Lemma \ref{lem:symmetry_P} and Formula \eqref{eq:a}. Therefore the fixed-point sets of $h,c \in \Isom(X)$ are contained in the interior of $P \subset X$.

Since $H$ and $C$ define a bicolouring of $\partial X$ (Proposition \ref{prop:X}) and the gluing maps $h$ and $c$ are two distinct, commuting, fixed-point free, isometric involutions of $\partial X$, it follows that $M$ is a hyperbolic manifold: near the points corresponding to 
the internal points of the facets of $X$
because $h$ and $c$ are fixed-point free isometric involutions of $\partial X$, and near the points corresponding to 
the corners of $X$ because moreover $h$ and $c$ are distinct and commute (so each corner cycle has length 4 and trivial return map).

The manifold $M$ is orientable because 
$h$ and $c$ are orientation reversing by Lemma 
\ref{lem:symmetry_P} 
(recall that $X$ is orientable by Proposition \ref{prop:X}).

The orbifold Euler characteristic of $Q$ is $\chi^\orb(Q) = 1/60$ (for a quick check one may use 
\cite{coxiter}), and $M$ is tessellated into $5! = 120$ copies of $Q$. Therefore $\chi(M) = 120/60 = 2$. 

Finally, 
note that the gluing maps used to build $M$ from the copies of $P$ are induced by symmetries of $P$. Therefore $M$ covers the orbifold $P / \Isom(P) \cong Q / \Isom(Q)$ (recall Lemma \ref{lem:symmetry_P})
, and so $M$ and $Q$ are commensurable
. 
\end{proof}

We conclude the paper with some additional information.

\begin{rem} \label{rem:non-or}
One builds a few non-orientable manifolds with $\chi = 2$ in the same commensurability class, just by choosing  as gluing maps 
$h$ and $c$ other pairs of distinct, commuting, isometric involutions of $X$ with no fixed point on $\partial X$, without requiring they both reverse the orientation. Any two among $a$, $a \circ r_{12}$, $a \circ r_{34}$ and $a \circ r_{12} \circ r_{34}$ work.
\end{rem}

\subsection{More Kerckhoff--Storm polytopes} \label{sec:same-method} 

The method applies to two additional Coxeter polytopes of Kerckhoff and Storm's with the same combinatorics of $Q$, called $Q_{t_4}$ and $Q_{t_5}$ in \cite{KS}.
One lies in the commensurability class \eqref{it:1} and has $\chi^\orb = 5/192$, while the other one is non-arithmetic and has $\chi^\orb = 241/7200$ (see \cite[Theorem 13.2]{KS} and \cite[Propositions 3.22 and 4.25]{MR}).

The main difference is that now the 
pentagonal faces of $P$ have dihedral angle $\pi/2$ (resp. $2\pi/5$) in place of $2\pi/3$, therefore the link of the top vertex $V$ tessellates $S^3$ like the boundary of the regular 16-cell (resp. 600-cell) $R$ in place of the simplex $\Delta$. So, to build $X$, one has to glue 16 (resp. 600) copies of $P$ in place of 5. This time $\Sym_4$ acts on $P \subset X$ like a facet stabiliser in 
$\Isom(R)$, in place of $\Isom(\Delta) \cong \Sym_5$.

The analogous construction gives an orientable 
manifold $M$ with $\chi = 10$ (resp. $
482$). Moreover, in this case, $M$ has an orientation-preserving isometric involution $\iota$ without fixed points, so the quotient $M / \langle \iota \rangle$ is a twice-smaller orientable manifold. One gets $\iota$ by composing the map induced by $a$ with that induced by the antipodal map of 
$R$ (a symmetry that 
$\Delta$ does not enjoy!).

In conclusion, one gets two more cusped, orientable, hyperbolic 4-manifolds: one in the class \eqref{it:1} of $\matH^4/\mathrm{PO}(4,1;\mathbb Z)$ with $\chi = 5$, and one non-arithmetic with $\chi = 241$.
%

\end{document}